\documentclass[11pt]{article}
\usepackage{epic,latexsym,amssymb}
\usepackage{color}
\usepackage{tikz}
\usepackage{amsmath,amsthm}
\usepackage{dirtytalk}

\newtheorem{thm}{Theorem}
\newtheorem{lem}{Lemma}

\newtheorem{claim}{Claim}

\newtheorem{defn}{Definition}

\newcommand{\Z}{{\Z B}}

\let\oldenumerate\enumerate
\renewcommand{\enumerate}{
  \oldenumerate
  \setlength{\itemsep}{0pt}
  \setlength{\parskip}{0pt}
  \setlength{\parsep}{0pt}
}

\def\vertex(#1){\put(#1){\circle*{2}}}
\def\vertexo(#1){\put(#1){\circle{2}}}
\def\vert(#1){\put(#1){\circle*{1.5}}}
\def\verto(#1){\put(#1){\circle{1.5}}}
\def\lab(#1)#2{\put(#1){\makebox(0,0)[c]{#2}}}

\begin{document}

\title{ The $1$-nearly vertex independence number of a graph}

\author{Zekhaya B. Shozi \thanks{Research supported by University of KwaZulu-Natal.}\\
	School of Mathematics, Statistics \& Computer Science\\
	University of KwaZulu-Natal\\
	Durban, 4000 South Africa\\
\small \tt Email: zekhaya@aims.ac.za
}

\date{}
\maketitle

\begin{abstract}
 Let $G$  be a graph with vertex set $V(G)$ and edge set $E(G)$. A set $I_0(G) \subseteq V(G)$ is a vertex independent set if no two vertices in $I_0(G)$ are adjacent in $G$. We study $\alpha_1(G)$, which is the maximum cardinality of a set $I_1(G) \subseteq V(G)$ that contains exactly one pair of adjacent vertices of $G$. We call $I_1(G)$ a $1$-nearly vertex independent set of $G$ and $\alpha_1(G)$ a $1$-nearly vertex independence number of $G$. We provide some cases of explicit formulas for $\alpha_1$. Furthermore, we prove a tight lower (resp. upper) bound on $\alpha_1$ for graphs of order $n$. The extremal graphs that achieve equality on each bound are fully characterised.
 \end{abstract}

{\small \textbf{Keywords:} $1$-Nearly vertex independent set; $1$-Nearly vertex independence number; Good graph} \\
\indent {\small \textbf{AMS subject classification:} 05C69}
\newpage
	
\section{Introduction}

A simple and undirected graph $G$ is an ordered pair of sets $(V(G), E(G))$, where $V(G)$ is a nonempty set of elements called \emph{vertices} and $E(G)$ is a (possibly empty) set of $2$-element subsets of $V(G)$ called \emph{edges}. For convenience, we often write $uv$ instead of $\{u,v\}$ to represent the edge joining the vertices $u$ and $v$ in a graph $G$.  The number of vertices of a graph $G$ is the \emph{order} of $G$ and is denoted by $n$, while the number of edges of $G$ is a the \emph{size} of $G$ and is denoted by $m$.

A \emph{vertex independent set} of a graph $G$ is a set $I_0(G) \subseteq V(G)$ such that all elements of $I_0(G)$ are pairwise nonadjacent in $G$. The \emph{vertex independence number} of a graph $G$ is the maximum cardinality of a vertex independent set in $G$, and is denoted by $\alpha_0(G)$. The literature shows that several researchers have studied the vertex independence number of a graph, and the study itself dates back to the $1970$s. See the article by Cairo and Yair \cite{caro1979new}. Also, see the article by Lov{\'a}sz \cite{lovasz1982bounding}, where the bounds on the vertex independence number of a graph are studied. Several other researchers   studied the vertex independence number of a graph while imposing various graph restrictions such as considering only connected graphs \cite{rad2018new}, fixing the order and size of a graph \cite{harant2001independence}, and fixing the maximum degree of a graph \cite{jones1984independence, harant2008independence, kanj2013independence}, among others. The authors in \cite{frieze1990independence} studied the vertex independence number of a random graph.

A generalisation of the vertex independence number of a graph has been proposed and given the name of the $k$-independence number. The \emph{$k$-independence number} of a graph $G$ is defined as the maximum cardinality of a set all whose vertices are at distance at least $k+1$ in $G$. For further information on this topic, please see \cite{favoron1987k, bouchou2014k, abiad2019k}. In this paper we propose a new other generalisation of the vertex independence number of a graph. A graph $H$ is a \emph{subgraph} of a graph $G$ if $V(H)\subseteq V(G)$  and $E(H)\subseteq E(G)$. If furthermore 
$$
E(H)=E(G)\cap \{\{u,v\} \mid u,v\in V(G)\},
$$
then we say that $H$ is an \emph{induced subgraph} of $G$. This means that for a given set of vertices $V(H)$, it has all possible edges of $G$ that it can have. For an integer $k\ge 0$, we define a \emph{$k$-nearly vertex independent set} of a graph $G$ to be a set $I_k \subseteq V(G)$ such that the subgraph induced by $I_k$ in $G$ has size exactly $k$. We remark that a $0$-nearly vertex independent set is a vertex independent set. The \emph{$k$-nearly vertex independence number} of a graph $G$ is the maximum cardinality of a $k$-nearly vertex independent set in $G$, and is denoted by $\alpha_k(G)$. 

This paper particularly focuses on $\alpha_1$. At least for the classes of graphs that we investigated, the behavior of $\alpha_1$ is sometimes similar and sometimes different to that of $\alpha_0$. Among all graphs of order $n$, while the edgeless  graph $\overline{K_n}$ has the largest $\alpha_0$, it has the smallest $\alpha_1$ as $\alpha_1(\overline{K_n})=0$. Among all connected graphs of order $n \ge 2$, while the star $K_{1,n-1}$ has the largest $\alpha_0$, it has the smallest $\alpha_1$ as $\alpha_1(K_{1,n-1})=2$. The complete graph $K_n$, that has all possible edges, which has the smallest $\alpha_0$ also has the smallest $\alpha_1$ if $n\geq 2$.

The rest of the paper is structured as follows. We start with a preliminary section, Section \ref{preliminary}. Specifically, in Section \ref{recursive-formula} we provide a  recursive formula that we will use frequently throughout this paper, and in Section \ref{explicit-formulas} we provide explicit formulas for $\alpha_1$ of some classes of graphs. The main results are in Section \ref{Sec:Main}. There, we provide a characterisation of the graph that achieve the minimum (and maximum) $\alpha_1$ among all graphs of order $n$. The graph that achieves the minimum $\alpha_1$ among all graphs of order $n$ is the edgeless graph $\overline{K_n}$, while the graph that achieves the maximum $\alpha_1$ among all graphs of order $n$ is the graph with exactly one edge. We also provide a full characterisation of the family of the connected graphs of order $n$ that achieve the minimum $\alpha_1$. Lastly, we provide a characterisation of the connected graphs of order $n$ that achieve the maximum $\alpha_1$. There are two such connected graphs of order $n$; one is a unicyclic graph $U_{1,n-1}$ and the other one is a broom graph $B^3_n$.

\section{Preliminary}
\label{preliminary}

  For graph theory notation and terminology, we generally follow~\cite{henning2013total}. Let $G$ be a graph with vertex set $V(G)$, edge set $E(G)$, order $n = |V(G)|$ and size $m = |E(G)|$. We denote the degree of a vertex $v$ in $G$ by $\deg_Gv$. The complement of $G$ is denoted by $\overline{G}$ and defined as 
 $$\overline{G}=(V(G),\{uv \mid u,v\in V(G), u\neq v \text{ and }uv\notin E(G)\}).$$
 We define the \emph{open neighbourhood} of a vertex $v$ of a graph $G$ to be the set
 $$N_G(v) = \{ u \in V(G) \mid uv \in E(G) \},$$
 while its \emph{closed neighbourhood} is $N_G[v] = N_G(v) \cup \{v\}$.
 
 For a subset $S$ of vertices of a graph $G$, we denote by $G - S$ the graph obtained from $G$ by deleting the vertices in $S$ and all edges incident to them. If $S = \{v\}$, then we simply write $G - v$ rather than $G - \{v\}$. For any graphs $G$ and $H$, we define the \emph{join}
 $$
 G+H=(V(G)\cup V(H), E(G)\cup E(H)\cup\{uv \mid u\in V(G)\text{ and } v\in V(H) \})
 $$
 obtained by adding all possible edges between $G$ and $H$. 
 The subgraph of $G$ induced by the set $S$ is denoted by $\langle S \rangle_G$.
 
  We denote the path graph, cycle graph, wheel graph and complete graph on $n$ vertices by $P_n$, $C_n$, $W_n$ and $K_n$, respectively. A cycle graph $C_3$ is also called a \emph{triangle} in the literature. For positive integers $r$ and $s$, we denote by $K_{r,s}$ the complete bipartite graph with partite sets $X$ and $Y$ such that $|X|=r$ and $|Y|=s$. A complete bipartite graph $K_{1,n-1}$ is also called a \emph{star} in the literature.  Let $k\ge 2$ be an integer. A \emph{broom graph}, denoted $B_n^k$, is the tree of order $n$ obtained from the path, $P_k$, of order $k$ by adding $n-k$ new vertices and then joining them to exactly one end-vertex of $P_k$. Let $U_{1,n-1}$ be the graph obtained from the star $K_{1,n-1}$ by adding one more edge to make it a
  unicyclic graph.
  
  We use the standard notation $[k] = \{1,\ldots,k\}$.

\subsection{Recursive formula}
\label{recursive-formula}

For graph $G = (V(G), E(G))$, there exists an edge $e=uv$ of $G$ such that $\{u,v\}$ belongs to a maximum $1$-nearly vertex independent set of $G$. Thus, we have
\begin{align*}
    \alpha_1(G) = 2 + \alpha_0(G - N_G[u]\cup N_G[v]),
\end{align*}
where $2$ is the cardinality of a $1$-nearly vertex independent set, $\{u,v\}$, of $G$ formed by the two endpoints $u$ and $v$ of $e \in E(G)$ and $\alpha_0(G - N_G[u]\cup N_G[v])$ is the cardinality of a maximum $0$-nearly vertex independent set of the graph obtained from $G$ by deleting the vertices $u$ and $v$ along with all their neighbours. 

\subsection{Explicit formulas for $\alpha_1$ of some graphs}
\label{explicit-formulas}

In this section we present the explicit formulas for $\alpha_1$ of some famous classes of graphs such that the complete graph $K_n$, the path graph $P_n$, the cycle graph $C_n$ and the wheel graph $W_n$, of order $n$. 

Let $K_n$ be the complete graph of order $n$. Since every vertex of $K_n$ is adjacent to every other vertex of $K_n$, a maximum $1$-nearly vertex independent set of $K_n$ is $\{u,v\}$, where $u$ and $v$ are the endpoints of some edge $e\in E(K_n)$. Thus, $\alpha_1(G) = |\{u,v\}|=2$. 

We state, without proof, the following well-known result for the $0$-nearly vertex independence number of a path with $n$ vertices.

\begin{thm}[\rm \cite{rad2016note}]
    \label{thm:alpha-0-of-a-path}
    For a path $P_n$ of order $n\ge 1$, we have 
    \begin{align*}
        \alpha_0(P_n) = \left \lfloor \frac{n+1}{2} \right \rfloor.
    \end{align*}
\end{thm}

For an integer $n\ge 2$, let $P_n=u_1, u_2, \ldots, u_n$ be a path of order $n$. Then, we can find a maximum $1$-nearly vertex independent set of $P_n$ that consists of the two vertices $u_1$ and $u_2$ as well as the $\alpha_0(P_{n-3})$ vertices that form a maximum $0$-nearly vertex independent set of $P_{n-3}$. Thus, we have
\begin{align*}
    \alpha_1(G) = 2 + \alpha_0(P_{n-3}) = 2 + \left \lfloor \frac{n-3 + 1}{2} \right \rfloor = \left \lfloor 2 + \frac{n-2}{2} \right \rfloor = \left \lfloor \frac{n+2}{2} \right \rfloor.
\end{align*}

For an integer $n\ge 3$, let $C_n= u_1, u_2, \ldots, u_n, u_1$ be a cycle of order $n$. Then, we can find a maximum $1$-nearly vertex independent set of $C_n$ that consists of the two vertices $u_1$ and $u_2$ as well as the $\alpha_0(P_{n-4})$ vertices that form a maximum $0$-nearly vertex independent set of $P_{n-4}$. Thus, we have
\begin{align*}
    \alpha_1(G) = 2 + \alpha_0(P_{n-4}) = 2 + \left \lfloor \frac{n-4 + 1}{2} \right \rfloor = \left \lfloor 2 + \frac{n-3}{2} \right \rfloor = \left \lfloor \frac{n+1}{2} \right \rfloor.
\end{align*}

For an integer $n\ge 4$, let $W_n = C_n + K_1$, where $C_n= u_1, u_2, \ldots, u_{n-1}, u_1$ is a cycle of order $n-1$. Then, we can find a maximum $1$-nearly vertex independent set of $W_n$ that consists of the two vertices $u_1$ and $u_2$ as well as the $\alpha_0(P_{n-5})$ vertices that form a maximum $0$-nearly vertex independent set of $P_{n-5}$. Thus, we have
\begin{align*}
    \alpha_1(G) = 2 + \alpha_0(P_{n-5}) = 2 + \left \lfloor \frac{n-5 + 1}{2} \right \rfloor = \left \lfloor 2 + \frac{n-4}{2} \right \rfloor = \left \lfloor \frac{n}{2} \right \rfloor.
\end{align*}

\section{Main result}
\label{Sec:Main}

In this section we present some bounds on the $1$-nearly vertex independence number of a graph $G$. In particular, we present a tight lower (and upper) bound on the  $1$-nearly vertex independence number of a general graph $G$. Thereafter, we present a tight lower (and upper) bound on the  $1$-nearly vertex independence number of a connected graph $G$. The graphs that achieve equality on each of these bounds are characterised. 

\subsection{Tight lower (and upper) bound for general graphs}

We begin by proving the following important lemmas.

\begin{lem}
    \label{lem:if-alpha-1-min-then-size-0}
    Let $G$ be a graph of order $n$. If $G$ has minimum $\alpha_1$, then $G$ has size $m=0$.
\end{lem}

\begin{proof}
    Let $G$ be a graph of order $n$. Suppose, to the contrary, that $\alpha_1(G)$ is minimum and $G$ has size $m \ge 1$. We will show that there exists a graph $G'$ with $n$ vertices such that $\alpha_1(G') < \alpha_1(G)$. Let $E(G) = \{ e_i \mid 1 \le i \le m \}$, and consider the graph $G' = G - E(G)$. Note that $G'$ is the edgeless graph of order $n$ with  $1$-nearly vertex independence number
    \begin{align*}
        \alpha_1(G') = 0 < 2 \le \alpha_1(G).
    \end{align*}
    However, this contradicts the fact that $G$ has minimum $\alpha_1$.
\end{proof}

\begin{lem}
    \label{lem:if-alpha-1-max-then-size-1}
    Let $G$ be a graph of order $n$. If $G$ has maximum $\alpha_1$, then $G$ has size $m=1$.
\end{lem}

\begin{proof}
    Let $G$ be a graph of order $n$. Suppose, to the contrary, that $\alpha_1(G)$ is maximum and $G$ has size $m \ne 2$. If $m=0$, then $\alpha_1(G) =0$ and therefore $\alpha_1(G)$ is not maximum. Hence, we may assume that $m\ge 2$. We will show that there exists a graph $G'$ with $n$ vertices such that $\alpha_1(G') > \alpha_1(G)$. Let $E(G) = \{ e_i \mid 1 \le i \le m \}$, and let $I_1(G)$ be a maximum $1$-nearly vertex independent set in $G$ such that $\{u,v\} \subseteq I_1(G)$ and $e_1 = uv$. Now consider the graph $G' = G - \{e_i \mid 2 \le i \le m\}$. Note that $G'$ has order $n$. Thus, a maximum $1$-nearly vertex independent set in $G'$ is given by $I_1(G') = I_1(G) \cup \{ w \in V(G) \mid w \notin I_1(G) \}$, implying that
    \begin{align*}
        \alpha_1(G') = |I_1(G')| &= |I_1(G)| + |\{ w \in V(G) \mid w \notin I_1(G) \}|\\
        &=\alpha_1(G) + n - \alpha_1(G)\\
        &=n\\
        &>\alpha_1(G).
    \end{align*}
    However, this contradicts the fact that $G$ has maximum $\alpha_1$.
\end{proof}

\begin{thm}
    \label{thm:bounds-general-graphs}
    If $G$ is a graph of order $n$, then 
    \begin{align*}
        0 \le \alpha_1(G) \le n,
    \end{align*}
    where the lower bound is achieved if and only if $G \cong \overline{K_n}$ and the upper bound is achieved if and only if $G \cong K_2 \cup (n-2)K_1$.
\end{thm}

\begin{proof}
    Let $G$ be a graph of order $n$. The lower bound $\alpha_1(G) \ge 0$ follows from the fact that if $G \cong \overline{K_n}$, then $\alpha_1(G) = 0$. The proof of a characterisation of the graph that achieves the lower bound follows immediately from the proof of Lemma \ref{lem:if-alpha-1-min-then-size-0}. Similarly, the upper bound $\alpha_1(G) \le n$ follows from the fact that if $G \cong K_2 \cup (n-2)K_1$, then $\alpha_1(G) = n$. The proof of a characterisation of the graph that achieves the upper bound follows immediately from the proof of Lemma \ref{lem:if-alpha-1-max-then-size-1}. 
\end{proof}

\subsection{Tight lower  bound for connected graphs}

In this subsection we present a tight lower bound on the $1$-nearly vertex independence number of a connected graph of order $n$. Furthermore, we characterise the family of the connected graphs of order $n$ that achieve the minimum $1$-nearly vertex independence number. 

\begin{defn}
\label{Def:Good}
    Let $G = (V(G), E(G))$ be a graph with vertex set $V(G)$ and edge set $E(G)$. If $e=uv \in E(G)$, then $e$ is a good edge if $N_G[u]\cup N_G[v] = V(G)$. The graph $G$ is a good graph if for every edge $e \in E(G)$, $e$ is a good edge. Let
    $$\mathcal{H} = \{ G \mid G \text{ is a good graph} \}.$$
\end{defn}
It follows from the definition that a good graph has to be connected.

\begin{thm}
    If $G$ is a connected graph of order $n\ge 2$, then
    \begin{align}
        \label{if-G-is-a-connected-graph-of-order-n-size-m-then-alpha1G-atleast-2}
        \alpha_1(G) \ge 2,
    \end{align}
    with equality if and only if $G \in \mathcal{H}$.
\end{thm}

\begin{proof}
Let $G$ be a connected graph of order $n\ge 2$. Since $G$ is connected, there exists at least one edge $uv$ in $G$. Note that $I_1(G) = \{u,v\}$ forms a $1$-nearly vertex independent set of $G$. Thus, 
\begin{align*}
    \alpha_1(G) \ge |I_1(G)| = 2,
\end{align*}
and thereby proving Inequality (\ref{if-G-is-a-connected-graph-of-order-n-size-m-then-alpha1G-atleast-2}).

 If $G$ is an $(n,m)$-graph that is not good, then there is an edge $uv$ of $G$ that is not a good edge. That is, $V(G)\setminus(N_G[u]\cup N_G[v])\neq \emptyset$. Thus, if $\{u,v\}$ is a $1$-nearly vertex independent set of $G$, then the set $\{u,v,z\}$ is also a $1$-nearly vertex independent set of $G$ for any $z\in V(G)\setminus (N_G[u]\cup N_G[v])$. Hence, we have
 \begin{align*}
     \alpha_1(G) \ge |\{u,v,z\}| = 3 > |\{u,v\}| = 2,
 \end{align*}
 completing the proof.
\end{proof}

Definition \ref{Def:Good} already gives a full characterisation of the family $\mathcal{H}$. However, using that description, it is still hard to imagine how the structure of an element of $\mathcal{H}$ with a large number of vertices should look like. The rest of this subsection is a further investigation of $\mathcal{H}$ aiming to understand the structures of its elements.

The proofs of Lemma \ref{Lem:ClosedH}, Lemma \ref{Lem:ExtendH} and Theorem \ref{thm:a-family-of-minimal-graphs} stated below can be found in \cite{andriantiana2023number}, but are also presented in this paper for completeness.

\begin{lem}
\label{Lem:ClosedH}
    The family $\mathcal{H}$ is closed under the join operation.
\end{lem}

\begin{proof}
    Let $G_1 = (V(G_1), E(G_1))$ and $G_2 = (V(G_2), E(G_2))$ be two graphs in $\mathcal{H}$. Then $G_1$ and $G_2$ are good graphs.  We want to show that $G_1+G_2\in \mathcal{H}$. Since $G_1$ and $G_2$ are good graphs, we have $N_{G_1}[s] \cup N_{G_1}[t] = V(G_1)$ for all $st\in E(G_1)$, and $N_{G_2}[u] \cup N_{G_2}[v] = V(G_2)$ for all $uv\in E(G_2)$. Let $wx \in E(G_1+G_2)$. 
    Then either $wx \in E(G_1)$ or $wx \in E(G_2)$ or $w \in V(G_1)$ and $x\in V(G_2)$. If $wx \in E(G_1)$, then, since $G_1$ is a good graph, $wx$ is a good edge in $G_1$. By the definition of the join operation, $N_{G_1+G_2}[w] = N_{G_1}[w] \cup V(G_2)$ and $N_{G_1+G_2}[x] = N_{G_1}[x] \cup V(G_2)$. Thus, $$N_{G_1+G_2}[w] \cup N_{G_1+G_2}[x] = N_{G_1}[w] \cup N_{G_1}[x] \cup V(G_2) = V(G_1)\cup V(G_2) = V(G_1+G_2).$$ By the same reasoning, if $wx \in E(G_2)$, then, since $G_2$ is a good graph, $wx$ is a good edge in $G_2$, and by the definition of the join operation, $wx$ is also a good edge in $G_1+G_2$.  Hence, we assume that $w \in V(G_1)$ and $x \in V(G_2)$. By the definition of the join operation, $N_{G_1+G_2}[w] = N_{G_1}[w] \cup V(G_2)$ and $N_{G_1+G_2}[x] = N_{G_2}[x] \cup V(G_1)$. Since $N_{G_1}[w] \subseteq V(G_1)$ and $N_{G_2}[x] \subseteq V(G_2)$, we have $$N_{G_1+G_2}[w] \cup N_{G_1+G_2}[x] = V(G_1)\cup V(G_2)=V(G_1+G_2),$$ implying that $wx$ is also a good edge in $G_1+G_2$. Since $wx$ was arbitrarily chosen, the graph $G_1+G_2$ is a good graph,  and thus $G_1 + G_2 \in \mathcal{H}$. 
\end{proof}

\begin{lem}
\label{Lem:ExtendH}
    If $G \in \mathcal{H}$, then for any integer $\ell \ge 1$, $G + \overline{K_\ell} \in \mathcal{H}$, where $\overline{K_\ell}$ is an edgeless graph of order $\ell$.
\end{lem}

\begin{proof}
    Let $G \in \mathcal{H}$. Then $G$ is a good graph, and so $N_G[u] \cup N_G[v] = V(G)$ for all $uv \in E(G)$. Let $\overline{K_\ell}$ be the empty graph of order $\ell$ with vertex set $V(\overline{K_\ell}) = \{w_1, w_2, \ldots, w_\ell\}$. By the join operation, $N_{G+\overline{K_\ell}} [w_i] = V(G)\cup \{w_i\}$ for all $i \in [\ell]$. Let $vx \in E(G+\overline{K_\ell})$. 
    Since $\overline{K_\ell}$ is an edgeless graph, $vx \notin E(\overline{K_\ell})$. If $vx \in E(G)$, then, since $G$ is a good graph, $vx$ is a good edge in $G$. By the definition of the join operation, $N_{G+\overline{K_\ell}}[v] = N_{G}[v] \cup V(\overline{K_\ell})$ and $N_{G+\overline{K_\ell}}[x] = N_{G}[x] \cup V(\overline{K_\ell})$. Thus, $$N_{G+\overline{K_\ell}}[v] \cup N_{G+\overline{K_\ell}}[x] =  N_{G}[v] \cup N_{G}[x] \cup V(\overline{K_\ell}) = V(G) \cup V(\overline{K_\ell}) = V(G + \overline{K_\ell}),$$ implying that $vx$ is also a good edge in $G+\overline{K_\ell}$. Hence, we assume that $v\in V(G)$ and $x\in V(\overline{K_\ell})$. Thus, $x=w_i$ for some $i\in [\ell]$. By the definition of the join operation, $N_{G+\overline{K_\ell}} [v] = N_G[v] \cup V(\overline{K_\ell})$. Since $N_G[v] \subseteq V(G)$, we have $$N_{G+\overline{K_\ell}} [v] \cup N_{G+\overline{K_\ell}} [x] = N_{G+\overline{K_\ell}} [v] \cup N_{G+\overline{K_\ell}} [w_i] = V(G)  \cup V(\overline{K_\ell})= V(G+\overline{K_\ell}),$$ implying that $vx = vw_i$ is a good edge in $G+\overline{K_\ell}$. Since $vx$ was arbitrarily chosen, the graph $G+\overline{K_\ell}$ is a good graph, and thus $G+\overline{K_\ell} \in \mathcal{H}$.
\end{proof}

Let $\mathcal{H}_1=\{K_1\}\cup \{K_{r,s}\mid r,s\in \mathbb{N}\}$. For any integer $k\geq 2$, we define
$$\mathcal{H}_k=\{K+H\mid K,H\in \mathcal{H}_{k-1}\}\cup \{G+\overline{K_\ell}\mid G\in \mathcal{H}_{k-1} \text{ and }\ell\in\mathbb{N} \}.$$

\begin{thm}
\label{thm:a-family-of-minimal-graphs}
$$\mathcal{H}=\bigcup_{k\in\mathbb{N}}\mathcal{H}_k.$$
\end{thm}

\begin{proof}
It is easy to check that $\mathcal{H}_1\subseteq \mathcal{H}$. Also, by Lemma \ref{Lem:ClosedH} and Lemma \ref{Lem:ExtendH}, it follows that $\bigcup_{k\in\mathbb{N}}\mathcal{H}_k\subseteq \mathcal{H}$. Thus, it remains for us to prove the reverse inclusion. Suppose that $G$ is a good graph. Then $G$ is connected. If $G$ has only one vertex, then $G \in \mathcal{H}_1 \subseteq \mathcal{H}$. Hence, we may assume that $G$ has more than one vertex. 
Thus $G$ has at least one edge. We now consider the subgraph $H$ of $G$ with largest number of vertices, and such that $H=K+L$ for some induced subgraphs $K$ and $L$ of $G$. Suppose that $H$ is not $G$. Since $G$ is connected, there is a vertex $u$ in $G-H$ that is adjacent to a vertex $v$ in $H$. Without loss of generality, we can assume that $v\in V(K)$. Then, the edge $uv$ has to be a good edge; that is, $V(H) = V(K)\cup V(L)\subseteq N_G[u]\cup N_G[v]=V(G)$.

If there is a vertex $w\in V(K)$ that is not adjacent to $u$, then every vertex in $L$ has to be adjacent to $u$ (because every edge $ws$ for any $s\in V(L)$ has to be a good edge). In this case, we could add $u$ to $K$ and obtain a bigger subgraph of $G$, namely $H'=\langle V(K)\cup\{u\}\rangle_G+L$. However, this contradicts our choice of $H$. Hence, we may assume that all the vertices in $K$ are adjacent to $u$. In this case, we can again add $u$ to $L$ and have a bigger subgraph of $G$, namely $H''=K+\langle V(L)\cup\{u\}\rangle_G$. Once again, this is a contradiction to our choice of $H$. Hence, we must have $V(H)=V(G)$.

It is only left to prove that each of $K$ and $L$ is either an edgeless graph or a good graph.
It is sufficient to prove that if $K$ (or $L$) has an edge then it is a good graph. Suppose that $xy$ is an edge in $K$. Then it has to be a good edge in $G$; that is, $N_G[x]\cup N_G[y]=V(G)$. Thus, $N_K[x]\cup N_K[y]= (N_G[x]\cup N_G[y])\cap V(K)=V(G)\cap V(K)=V(K)$, implying that $xy$  also a good edge in $K$. 
\end{proof}

\subsection{Tight upper  bound for connected graphs}

In this subsection we present a tight upper bound on the $1$-nearly vertex independence number of a connected graph of order $n$. Furthermore, we characterise the two connected graphs of order $n$ that achieve the maximum $1$-nearly vertex independence number. 

\begin{thm}
	 \label{thm:if-G-is-a-connected-graph-of-order-n-then-alpha1G-atmost-n-1}
    If $G$ is a connected graph of order $n\ge 3$, then
    \begin{align*}
        \alpha_1(G) \le n-1,
    \end{align*}
    with equality if and only if $G \cong B^3_n$ or $G\cong U_{1,n-1}$.
\end{thm}

\begin{proof}
    Note that $\alpha_1(B^3_n) = \alpha_1(U_{1,n-1}) =n-1 $. We proceed by induction on the order $n\ge 3$ of a connected graph $G$. If $n=3$, then $G \cong P_3$ in which case $\alpha_1(G) = 2 =n-1$ or $G \cong U_{1,2}$ in which case $\alpha_1(G) = 2=n-1$. Suppose the result is true for all connected graphs of order $3\le k < n$, and let $G$ be a connected graph of order $n=k$ such that $\alpha_1(G) = n-1$. First, we show that if $G$ contains a cycle, then $G \cong B^3_n$. Suppose that $G$ contains a cycle, $C_\ell$, of order $\ell$. We proceed with the following series of claims.

    \begin{claim}
    \label{claim:ell-is-3}
         $\ell =3$.
    \end{claim}

    \begin{proof}
        Suppose, to the contrary, that $\ell \ge 4$, and let $C_\ell = v_1, v_2, \ldots, v_\ell, v_1$ be a largest cycle of $G$. Now consider the graph $G' = G - N_G[v_1] \cup N_G[v_2]$. Since $\ell \ge 4$, we note that $G'$ is a graph of order 
        $n' = n - \deg_G v_1 - \deg_G v_2$. Thus, if $I_0(G')$ is a maximum $0$-nearly vertex independent set in $G'$, then $I_1(G) = \{v_1, v_2\} \cup I_0(G')$ is a maximum $1$-nearly vertex independent set in $G$. Hence,
        \begin{align*}
            \alpha_1(G) &= |\{v_1, v_2\}| + |I_0(G')|\\
            &=2 + \alpha_0(G')\\
            &\le 2 +  n -  \deg_G v_1 - \deg_G v_2\\
            &= (n-1) + 3 - \deg_G v_1 - \deg_G v_2.\\
            &< n-1 & \text{ since } \deg_G v_i \ge 2 \text{ for } i \in [2],
        \end{align*}
        a contradiction.
    \end{proof}

     \begin{claim}
    \label{claim:G-contains-exactly-one-cycle-which-is-a-triangle}
         $G$ contains exactly one cycle which is a triangle.
    \end{claim}

    \begin{proof}
        Suppose, to the contrary, that $G$ contains at least two cycles. By Claim \ref{claim:ell-is-3}, every  cycle of $G$ is a triangle $C_3$. Note that any two triangles of $G$ intersect at at most one vertex, otherwise there would be a bigger cycle in $G$. Let $C_{3,1} = v_{1,1}, v_{2,1}, v_{3,1}, v_{1,1}$ and $C_{3,2} = v_{1,2}, v_{2,2}, v_{3,2}, v_{1,2}$ be any two triangles of $G$. We now consider the following two cases.

        \textbf{Case I: } $V(C_{3,1}) \cap V(C_{3,2}) \ne \emptyset$.

        There exists a vertex $v \in V(G)$ such that $v \in V(C_{3,1})$ and $v \in V(C_{3,2})$. Without any loss of generality, we may assume that $v_{1,1} = v_{1,2}=v$. Now consider the graph $G'= G- N_G[v_{2,1}] \cup N_G[v_{3,1}]$. Note that $G'$ has order $n'=n -\deg_G v_{2,1} -\deg_G v_{3,1} +1$. Since $G'$ contains an edge $v_{2,2}v_{3,2}$, any maximum $0$-nearly vertex independent set of $G'$ has cardinality at most $n'-1$. Thus, if $I_0(G')$ is a maximum $0$-nearly vertex independent set in $G'$, then $I_1(G) = \{v_{2,1}, v_{3,1}\} \cup I_0(G')$ is a maximum $1$-nearly vertex independent set in $G$. Hence,
        \begin{align*}
            \alpha_1(G) &= |\{v_{2,1}, v_{3,1}\}| + |I_0(G')|\\
            &=2 + \alpha_0(G')\\
            &\le 2 + n' -1\\
            &= 2 +  n -\deg_G v_{2,1} -\deg_G v_{3,1} + 1 -1\\
            &= (n-1) + 3 -\deg_G v_{2,1} -\deg_G v_{3,1}\\
            &< n-1  & \text{ since } \deg_G v_{i,1} \ge 2 \text{ for } 2\le i \le 3,
        \end{align*}
        a contradiction.

        \textbf{Case II: } $V(C_{3,1}) \cap V(C_{3,2}) = \emptyset$.

        If this is the case, then there must exist a an edge $e = v_{2,1}v_{3,1}$ of $C_{3,1}$ such that $N_G[v_{2,1}] \cup N_G[v_{3,1}] \nsubseteq V(C_{3,1})$. Now consider the graph $G'= G- N_G[v_{2,1}] \cup N_G[v_{3,1}]$. Note that $G'$ has order $n'=n -\deg_G v_{2,1} -\deg_G v_{3,1} +1$. Since $G'$ contains a triangle $C_{3,2}$, any maximum $0$-nearly vertex independent set of $G'$ has cardinality at most $n'-2$. Thus, if $I_0(G')$ is a maximum $0$-nearly vertex independent set in $G'$, then $I_1(G) = \{v_{2,1}, v_{3,1}\} \cup I_0(G')$ is a maximum $1$-nearly vertex independent set in $G$. Hence,
        \begin{align*}
        	\alpha_1(G) &= |\{v_{2,1}, v_{3,1}\}| + |I_0(G')|\\
        	&=2 + \alpha_0(G')\\
        	&\le 2 + n' -2\\
        	&= n -\deg_G v_{2,1} -\deg_G v_{3,1} + 1\\
        	&= (n-1) + 2 -\deg_G v_{2,1} -\deg_G v_{3,1}\\
        	&< n-1  & \text{ since } \deg_G v_{i,1} \ge 2 \text{ for } 2\le i \le 3,
        \end{align*}
        a contradiction.
    \end{proof}
    
    \begin{claim}
    	\label{claim:at-most-one-vertex-of-the-triangle-of-G-has-degree-at-least-3}
    	At most one vertex of the triangle of $G$ has degree at least $3$.
    \end{claim}
    
    \begin{proof}
    	Let $C_3=v_1, v_2, v_3,v_1$ be the triangle of $G$. Suppose, to the contrary, that there exists at least two vertices of $C_3$ that have degrees at least $3$. Let $v_1$ and $v_2$ be two vertices of $C_3$ such that $\deg_G v_i \ge 3$ for $i \in [2]$. Now consider the graph $G' = G - N_G[v_1] \cup N_G[v_2]$. Note that $G'$ is a graph of order 
    	$n' = n - \deg_G v_1 -\deg_Gv_2 + 1$. Thus, if $I_0(G')$ is a maximum $0$-nearly vertex independent set in $G'$, then $I_1(G) = \{v_1, v_2\} \cup I_0(G')$ is a maximum $1$-nearly vertex independent set in $G$. Hence,
    	\begin{align*}
    		\alpha_1(G) &= |\{v_1, v_2\}| + |I_0(G')\\
    		&=2 + \alpha_0(G')\\
    		&\le 2 +  n - \deg_G v_1 -\deg_Gv_2 + 1\\
    		&= (n-1) + 4 - \deg_G v_1 -\deg_Gv_2\\
    		&< n-1 &\text{ since } \deg_G v_i \ge 3 \text{ for some } i \in [2],
    	\end{align*}
    	 a contradiction.
    \end{proof}
    
    \begin{claim}
    	\label{claim:every-vertex-of-degree-2-in-G-lie-on-a-cycle}
    Every vertex of degree $2$ in $G$ lies on the triangle of $G$ 
    \end{claim}
    
    \begin{proof}
    	Let $C_3 = v_1, v_2, v_3, v_1$ be the triangle of $G$. Suppose, to the contrary, that there exists a vertex $w \in V(G)$ such that $\deg_G w =2$ and $w \ne v_i$ for all $i\in [3]$. By Claim \ref{claim:at-most-one-vertex-of-the-triangle-of-G-has-degree-at-least-3},  since $G$ is connected and $w \ne v_i$ for all $i\in [3]$, there must exist exactly one vertex $v_i$ of $C_3$, for $i\in [3]$, such that $\deg_G v_i \ge 3$. Without any loss of generality, let $\deg_G v_3 \ge 3$. Then, by Claim \ref{claim:at-most-one-vertex-of-the-triangle-of-G-has-degree-at-least-3}, $\deg_G v_i = 2$ for $i\in [2]$.
    	
    	Let $N_G(w) = \{w_1, w_2\}$. Note that at most one vertex $w_i \in N_G(w)$ for $i\in [2]$ lie on $C_3$, otherwise there would be a bigger cycle in $G$. Let $w_2$ be a neighbour of $w$ that does not lie on $C_3$. Now consider the graph $G' = G - N_G[v_1] \cup N_G[v_2]$. Note that $G'$ is a graph of order 
    	$n' = n - \deg_G v_1 - \deg_Gv_2 + 1$. Since $G'$ contains an edge $ww_2$, any maximum $0$-nearly vertex independent set of $G'$ has cardinality at most $n'-1$. Thus, if $I_0(G')$ is a maximum $0$-nearly vertex independent set in $G'$, then $I_1(G) = \{v_1, v_2\} \cup I_0(G')$ is a maximum $1$-nearly vertex independent set in $G$. Hence,
    	\begin{align*}
    		\alpha_1(G) &= |\{v_1, v_2\}| + |I_0(G')|\\
    		&=2 + \alpha_0(G')\\
    		&\le 2 + n' -1\\
    		&= 2 + n - \deg_G v_1 - \deg_Gv_2 + 1 -1 \\
    		&= (n-1) + 3 - \deg_G v_1 - \deg_Gv_2\\
    		&< n-1, &\text{ since } \deg_G v_i = 2 \text{ for } i \in [2],
    	\end{align*}
     a contradiction.
    \end{proof}
    
    \begin{claim}
    	\label{claim:every-neigbour-of-a-vertex-of-degree-at-least-3-in-C-3-that-lies-outside-C-3-has-degree-exactly-1}
    	If $C_3 = v_1, v_2, v_3, v_1$ is the triangle of $G$ with $\deg_G v_3 \ge 3$, then every neighbour of $v_3$ that does not lie on $C_3$ has degree exactly $1$.
    \end{claim}
    
    \begin{proof}
    	Let $C_3 = v_1, v_2, v_3, v_1$ be the triangle of $G$ with $\deg_G v_3 \ge 3$. Let $N_G(v_3) \setminus \{v_1, v_2\} = \{w_1, w_2, \ldots, w_r\}$, where $r= \deg_G v_3 -2$, be the set of all neighbours of $v_3$ that do not lie on $C_3$. Suppose, to the contrary, that $\deg_G w_i \ge 2$ for some $i \in [r]$. Without any loss of generality, let $\deg_G w_1 \ge 2$. Note that $v_3$ is the only neighbour of $w_1$ that lies on $C_3$, otherwise there would be bigger cycle in $G$. Thus, $N_G(w_1) \setminus \{v_3\} = \{w_{1,1}, w_{2,1}, \ldots, w_{s,1}\}$, where $s=\deg_G w_1 -1$, is the set of all neighbours of $w_1$ that do not lie on $C_3$.
    	
    	 Now consider the graph $G' = G - N_G[v_1] \cup N_G[v_2]$. Note that $G'$ is a graph of order 
    	$n' = n - \deg_G v_1 - \deg_Gv_2 + 1$. Since $G'$ contains an edge $w_1w_{1,1}$, any maximum $0$-nearly vertex independent set of $G'$ has cardinality at most $n'-1$. Thus, if $I_0(G')$ is a maximum $0$-nearly vertex independent set in $G'$, then $I_1(G) = \{v_1, v_2\} \cup I_0(G')$ is a maximum $1$-nearly vertex independent set in $G$. Hence,
    	\begin{align*}
    		\alpha_1(G) &= |\{v_1, v_2\}| + |I_0(G')|\\
    		&=2 + \alpha_0(G')\\
    		&\le 2 + n' -1\\
    		&= 2 +  n - \deg_G v_1 - \deg_Gv_2 + 1 -1 \\
    		&=(n-1) + 3 - \deg_G v_1 - \deg_Gv_2\\
    		&< n-1, &\text{ since } \deg_G v_i = 2 \text{ for } i \in [2],
    	\end{align*}
    	a contradiction.
    \end{proof}
    
    Recall that the graph $G$ contains a cycle. By Claim \ref{claim:ell-is-3}, every cycle of $G$ is a triangle. By Claim \ref{claim:G-contains-exactly-one-cycle-which-is-a-triangle}, the graph $G$ contains exactly one cycle which is a triangle. By Claim \ref{claim:at-most-one-vertex-of-the-triangle-of-G-has-degree-at-least-3}, at most one vertex of the triangle of $G$ has degree at least $3$. By Claim \ref{claim:every-vertex-of-degree-2-in-G-lie-on-a-cycle}, every vertex of degree $2$ in $G$ lies on the triangle of $G$. Furthermore, by Claim \ref{claim:every-neigbour-of-a-vertex-of-degree-at-least-3-in-C-3-that-lies-outside-C-3-has-degree-exactly-1}, every neighbour of the vertex of degree at least $3$ that lies on the triangle of $G$ has degree exactly $1$. These properties of the graph $G$ such that $G$ contains a cycle and $\alpha_1(G)=n-1$ are sufficient to deduce that $G \cong U_{1, n-1}$.
    
    Hence, we may assume that $G$ does not contain a cycle. If $G$ does not contain a cycle, then, since $G$ is connected, $G$ is a tree. We now complete the proof of Theorem \ref{thm:if-G-is-a-connected-graph-of-order-n-then-alpha1G-atmost-n-1} with the following claim.
    
    \begin{claim}
    	\label{claim:G-has-max-degree-exactly-n-2}
    	$G$ has maximum degree $r= n-2$.
    \end{claim}
    
    \begin{proof}
    	Let $v$ be a vertex of maximum degree in $G$. Suppose, to the contrary, that $r=\deg_G v \ne n-2$. If $r= n-1$, then $G \cong K_{1,n-1}$ and thus $\alpha_1(G) =2$. Hence, we may assume that $r \le n-2$. Let $N_G(v)$ be the set of all neighbours of $v$ in $G$. Since $r \le n-2$, there must exist a vertex $v_i \in N_G(v)$ such that $\deg_G v_i \ge 2$.  Let $A = \{v_1, v_2, \ldots, v_p\} \subseteq  N_G(v)$ be the set of all neighbours of $v$ that have degrees at least $2$ in $G$. We now consider the following cases.
    	
    	\textbf{Case I: } $r\le n-3$ and  $p \ge 2$.
    	
    	Suppose that $r\le n-3$ and  $p \ge 2$. Without any loss of generality, let $v_{1,1}$ and $v_{1,2}$ be neighbours of $v_1$ and $v_2$ in $G$, respectively. Now consider the graph $G'= G- N_G[v_1] \cup N_G[v_{1,1}]$. Note that $G'$ has order $n'=n- \deg_G v_1 - \deg_G v_{1,1}$. Since $G'$ contains an edge $v_2v_{1,2}$, any maximum $0$-nearly vertex independent set of $G'$ has cardinality at most $n'-1$. Thus, if $I_0(G')$ is a maximum $0$-nearly vertex independent set in $G'$, then $I_1(G) = \{v_1, v_{1,1}\} \cup I_0(G')$ is a maximum $1$-nearly vertex independent set in $G$. Hence,
    	\begin{align*}
    		\alpha_1(G) &= |\{v_1, v_{1,1}\}| + |I_0(G')|\\
    		&=2 + \alpha_0(G')\\
    		&\le 2 + n' -1\\
    		&= 2 +n - \deg_G v_1 - \deg_G v_{1,1} - 1\\
    		&= (n-1) + 2 - \deg_G v_1 - \deg_G v_{1,1}\\
    		&<n-1 &\text{ since } \deg_G v_1 \ge 2
			\text{ and } \deg_G v_{1,1} \ge 2,
    	\end{align*}
    	a contradiction. 
    	
    	Hence, we may assume that $r\le n-3$ and $p=1$. Thus, $A = \{v_1\}$. 
    	
    	\textbf{Case II: } $r\le n-3$, $p = 1$ and $\deg_G v_1 \ge 3$.  
    	
    	Suppose that $r\le n-3$, $p \ge 1$ and $\deg_G v_1 \ge 3$. Since $\deg_G v_1 \ge 3$, there must exist two vertices $v_{1,1}$ and $v_{2,1}$ such that $v_{1,1} \ne v_{2,1} \ne v$ and $\{ v_{1,1} , v_{2,1} \} \subseteq N_G(v_1)$. Now consider the graph $G'= G- N_G[v_1] \cup N_G[v_{1,1}]$. Note that $G'$ has order $n'=n -\deg_G v_1 -\deg_G v_{1,1}$.  Thus, if $I_0(G')$ is a maximum $0$-nearly vertex independent set in $G'$, then $I_1(G) = \{v_1, v_{1,1}\} \cup I_0(G')$ is a maximum $1$-nearly vertex independent set in $G$. Hence,
    	\begin{align*}
    		\alpha_1(G) &= |\{v_1, v_{1,1}\}| + |I_0(G')|\\
    		&=2 + \alpha_0(G')\\
    		&\le 2 +  n -\deg_G v_1 -\deg_G v_{1,1}\\
    		&= (n-1) + 3 -\deg_G v_1 -\deg_G v_{1,1}\\
    		&< n-1  & \text{ since } \deg_G v_1 \ge 3,
    	\end{align*}
    	a contradiction.
    	
    	\textbf{Case III: } $r\le n-3$, $p = 1$ and $\deg_G v_1 = 2$.  
    	
    	Suppose that $r\le n-3$, $p = 1$ and $\deg_G v_1 = 2$.  In this case, the vertex $v_{1,1} \in N_G(v_1)$ must have degree at least $2$ in $G$. Let $v_{1,1}^*$ be a neighbour of $v_{1,1}$ in $G$ such that $v_{1,1}^* \ne v_1$. Now consider the graph $G'= G- N_G[v_{1,1}] \cup N_G[v_{1,1}^*]$. Note that $G'$ has order $n'=n -\deg_G v_{1,1} -\deg_G v_{1,1}^*$.   Since $G'$ contains a component that is a star, any maximum $0$-nearly vertex independent set of $G'$ has cardinality at most $n'-1$. Thus, if $I_0(G')$ is a maximum $0$-nearly vertex independent set in $G'$, then $I_1(G) = \{v_{1,1}, v_{1,1}^*\} \cup I_0(G')$ is a maximum $1$-nearly vertex independent set in $G$. Hence,
    	\begin{align*}
    		\alpha_1(G) &= |\{v_{1,1}, v_{1,1}^*\}| + |I_0(G')|\\
    		&=2 + \alpha_0(G')\\
    		&\le 2 + n' -1\\
    		&= 2 +  n -\deg_G v_{1,1} -\deg_G v_{1,1}^* -1\\
    		&= (n-1) + 2 -\deg_G v_{1,1} -\deg_G v_{1,1}^*\\
    		&< n-1  & \text{ since } \deg_G v_{1,1} \ge 2,
    	\end{align*}
    	a contradiction.
    	
    	Thus, $r=\deg_G v = n-2$, thereby completing the proof of Claim \ref{claim:G-has-max-degree-exactly-n-2}.
    \end{proof}

    By Claim \ref{claim:G-has-max-degree-exactly-n-2} Case I, there is exactly one neighbour $v_1$ of $v$ that has degree at least $2$ in $G$. By Claim \ref{claim:G-has-max-degree-exactly-n-2} Case II, the degree of $v_1$ in $G$ is exactly $2$. By Claim \ref{claim:G-has-max-degree-exactly-n-2} Case III, the vertex $v_{1,1} \in N_G(v_1) \setminus \{v\}$ has degree exactly $1$ in $G$. Thus, the graph $G$ is a tree of order $n$ with degree sequence $(n-2, 2, 1,\ldots, 1)$, and so $G \cong U_{1,n-1}$. This completes the proof of Theorem \ref{thm:if-G-is-a-connected-graph-of-order-n-then-alpha1G-atmost-n-1}.
\end{proof}

\newpage
\bibliographystyle{abbrv} 
\bibliography{references}

\end{document}